\documentclass[11pt]{article}
\usepackage[dvips]{epsfig}
\usepackage{amsmath,amssymb,euscript,graphicx,amsfonts,verbatim}
\usepackage{enumerate,color}
\usepackage{graphicx}
\usepackage{hyperref}

\usepackage{graphicx}

\newtheorem{theorem}{Theorem}[section]
\newtheorem{proposition}[theorem]{Proposition}
\newtheorem{lemma}[theorem]{Lemma}
\newtheorem{corollary}[theorem]{Corollary}
\newtheorem{question}[theorem]{Question}
\newtheorem{conjecture}[theorem]{Conjecture}

\newtheorem{example}[theorem]{Example}

\newtheorem{remark}[theorem]{Remark}

\newenvironment{proof}{{\noindent \sc Proof. }}{\hfill $\Qed$\\}

\newcommand{\la}{\langle}
\newcommand{\ra}{\rangle}

\newcommand{\Qed}{\rule{2.5mm}{3mm}}

\newcommand{\Cay}{\hbox{{\rm Cay}}}

\newcommand{\ZZ}{\mathbb{Z}}

\newcommand{\B}{{\cal{B}}}

\newcommand{\G}{{\overline{G}}}
\newcommand{\F}{{\cal{F}}}

\newcounter{case}

\renewcommand{\thecase}{\arabic{case}}

\newcounter{subcase}

\numberwithin{subcase}{case}

\newenvironment{proofT2}{{\noindent \sc Proof of Theorem~\ref{the:main2}.}}{\hfill $\Qed$ \\}

\newenvironment{proofT3}{{\noindent \sc Proof of Theorem~\ref{the:main3}.}}{\hfill $\Qed$ \\}
\begin{document}


\begin{center}
{\bf\large INTERSECTION DENSITY OF TRANSITIVE GROUPS\\ 
OF CERTAIN DEGREES} \\ [+4ex]
Ademir Hujdurovi\'c{\small$^{a,b,}$}\footnotemark,  
Klavdija Kutnar{\small$^{a,b,}$}\footnotemark$^{,*}$,  
\addtocounter{footnote}{0} 
Dragan Maru\v si\v c{\small$^{a, b, c,}$}\footnotemark 
\ and
\v Stefko Miklavi\v c{\small$^{a, b, c,}$}\footnotemark 
\\ [+2ex]
{\it \small 
$^a$University of Primorska, UP IAM, Muzejski trg 2, 6000 Koper, Slovenia\\
$^b$University of Primorska, UP FAMNIT, Glagolja\v ska 8, 6000 Koper, Slovenia\\
$^c$IMFM, Jadranska 19, 1000 Ljubljana, Slovenia}
\end{center}

\addtocounter{footnote}{-3}
\footnotetext{The work of Ademir Hujdurovi\'c  is supported in part by the Slovenian Research Agency (research program P1-0404 and research projects
N1-0062, J1-9110, N1-0102, J1-1691, J1-1694, J1-1695, N1-0140, 
N1-0159 and J1-2451).}
\addtocounter{footnote}{1}
\footnotetext{The work of Klavdija Kutnar  is supported in part by the Slovenian Research Agency (research program P1-0285 and research projects
N1-0062, J1-9110, J1-9186, J1-1695, J1-1715, N1-0140, J1-2451 and J1-2481).}
\addtocounter{footnote}{1}
\footnotetext{The work of Dragan Maru\v si\v c is supported in part by the Slovenian Research Agency (I0-0035, research program P1-0285
and research projects N1-0062,  J1-9108, 
J1-1694, J1-1695, N1-0140 and J1-2451).}
\addtocounter{footnote}{1}
\footnotetext{
The work of \v Stefko Miklavi\v c is supported in part by the Slovenian Research Agency (research program P1-0285
and research projects N1-0062,  J1-9110, 
J1-1695, N1-0140, N1-0159 and J1-2451).

~*Corresponding author e-mail:~klavdija.kutnar@upr.si}

\begin{abstract}
Two elements $g$ and $h$ of a  permutation group $G$  acting on a set $V$ are said to be {\em intersecting} if $g(v) = h(v)$ for some $v \in V$.
More generally, a subset ${\cal F}$ of $G$ is an {\em intersecting set}
if every pair of elements of ${\cal F}$ is intersecting.
The {\em intersection density} $\rho(G)$ of a transitive permutation
group $G$ is the maximum value of the quotient $|{\cal F}|/|G_v|$ where ${\cal F}$ runs over all intersecting sets in $G$ and $G_v$ is a stabilizer
of $v\in V$.  In this paper the intersection density of transitive groups of degree twice a prime is determined, and proved to be either $1$ or $2$. In addition, it is proved that the intersection density of transitive groups of prime power degree  is $1$.
\end{abstract}

\begin{quotation}
\end{quotation}

\begin{quotation}
\noindent 
{\em Math. Subj. Class.:} 05C25, 20B25.
\end{quotation}


\section{Introductory remarks}
\label{sec:intro}
\noindent

For a finite set $V$ let   $\mathrm{Sym}(V)$ and  $\mathrm{Alt}(V)$ 
denote the corresponding symmetric group
and alternating group on $V$. (Of course, if $|V|=n$ the 
standard notations $S_n$, $A_n$ apply.)
Let $G\le \mathrm{Sym}(V)$ be a permutation group acting on a set $V$.
Two elements $g,h\in G$ 
are said to be {\em intersecting} if $g(v) = h(v)$ for some $v \in V$.
Furthermore, a subset ${\cal F}$ of $G$ is an {\em intersecting set}
if every pair of elements of ${\cal F}$ is intersecting.
The {\em intersection density} $\rho({\cal F})$ of the intersecting set ${\cal F}$
is defined to be the quotient  
$$
\rho({\cal F})=\frac{|{\cal F}|}{\max_{v\in V}|G_v|},
$$ 
and the {\em intersection density} $\rho(G)$ of a group $G$,
 first defined by Li, Song and Pantangi in \cite{LSP}, 
is the maximum value of 
$\rho({\cal F})$ where ${\cal F}$ runs over all intersecting sets in 
$G$, that is,
$$
\rho(G)=\max\{\rho({\cal F})\colon {\cal F}\subseteq G, {\cal F} \textrm{ is intersecting}\} = \frac{\max\{|{\cal F}| \colon {\cal F}\subset G\textrm{ is intersecting}\}}{\max_{v\in V}|G_v|}.
$$
Observe that, since $G_v$ is an intersecting set in $G$,
we have $\rho(G)\ge 1$.  
Observe also that for  a transitive group $G$
acting on a set $V$ we have $\rho(G)= 1$
if and only if 
the maximum cardinality of the intersecting set is $|G|/|V|$, in which case
we say that $G$ has the {\em Erd\"os-Ko-Rado property} or 
{\em EKR-property} in short. 
Moreover, $G$ has the {\em strict-EKR-property} if 
the canonical intersecting sets are the only maximum intersecting sets of $G$,
where a {\em canonical} intersecting set is an intersecting set of the form
$gG_v$, $v\in V$ and $g\in G$. 

Following \cite{MRS21} we define ${\cal I}_n$ to be the set
of all intersection densities of transitive permutation groups
of degree $n$, that is,
$$
{\cal I}_n=\{\rho(G) \mid G\textrm{ transitive of degree } n\},
$$ 
and we let $I(n)$ to be the maximum value in ${\cal I}_n$.

Motivation for this paper comes from \cite[Conjectures~6.(3) and {\color{red}6.}(4)]{MRS21}
and \cite[Question~7.1]{R21}. 

\begin{conjecture}
\label{conj:pk} {\rm\cite[Conjecture~6.6(3)]{MRS21}}
 If $n$ is a prime power, then $I(n)=1$.
\end{conjecture}

\begin{conjecture}
\label{conj:2p} {\rm\cite[Conjecture~6.6(4)]{MRS21}}
 If $n=2p$ where $p$ is a prime, then $I(n)=2$.
\end{conjecture}

Conjecture~\ref{conj:2p} is settled in \cite{R21},
where an additional  problem regarding the possible values of 
intersection densities in ${\cal I}_{2p}$ was posed.

\begin{question}
\label{que:2p}
{\rm \cite[Question~6.1]{R21}}
Does there exist an odd prime $p$ and a transitive group $G$ of degree $2p$
such that $\rho(G)$ is not an integer?
\end{question}

In this paper we settle Conjecture~\ref{conj:pk} and 
give a negative answer 
to Question~\ref{que:2p} by obtaining 
a complete classification of intersection densities of transitive groups 
of degree twice a prime. 

\begin{theorem}
\label{the:main2}
For a transitive permutation group $G$ of prime power degree 
the intersection density $\rho(G)$ is equal to $1$. 
\end{theorem}

\begin{theorem}
\label{the:main3}
Let $G$ be a transitive permutation group of degree $2p$, where $p$ is a prime.
Then the intersection density $\rho(G)$ is  either  $1$ or $2$. More precisely, $\rho(G)=2$ if and only if either 
\begin{enumerate}[(i)]
\itemsep=0pt
\item $G\cong K\rtimes H$ acting on a set $V=\{x_i\colon i\in\ZZ_p\}\cup\{y_i\colon i\in\ZZ_p\}$
where $K\le E\cap \mathrm{Alt}(V)$,  
$E\cong\ZZ_2^p$ is an elementary abelian $2$-group 
generated by the involutions $\epsilon_i=(x_i\ y_i)$, $i\in\ZZ_p$, 
and $H=\la (x_0\ x_1\ \ldots \ x_{p-1})(y_0\ y_1\ \ldots\  y_{p-1}) \ra\cong\ZZ_p$,
or
\item $G\cong A_5$ acting on a $10$-element set of pairs 
of $\{1,2,3,4,5\}$.
\end{enumerate}
\end{theorem}


\section{Preliminaries}
\label{sec:pre}
\noindent

\subsection{(Im)primitivity of transitive permutation groups}
\label{ssec:definition}
\noindent

Let $G$ be a 
transitive permutation group $G$ acting on a set $V$.
A partition $\B$ of $V$ is called $G$-{\em invariant}
if the elements of $G$ permute the parts, the so called
{\em blocks} of $\B$, setwise.
If the trivial partitions $\{V\}$ and $\{\{v\}: v \in V\}$ 
are the only
$G$-invariant partitions of $V$, then $G$ is {\em primitive},
and is {\em imprimitive} otherwise.
In the latter case the corresponding
$G$-invariant partition will be referred to as the
{\em complete imprimitivity block system} of $G$.  
We say that $G$ is {\em doubly transitive} if given any two ordered pairs
$(u,v)$ and $(u',v')$ of  elements $u,v,u',v'\in V$, such that $u\ne v$ and $u'\ne v'$,
there exists an element  $g\in G$ such that $g(u,v)=(u',v')$.
Note that a doubly transitive group is primitive.
A primitive group which is not doubly transitive is called {\em simply primitive}.

The following result about normalizers of Sylow $p$-subgroups
in doubly transitive groups of prime degree will be needed
in the proof of Theorem~\ref{the:main3}.

\begin{lemma}\label{lem:prime}
Let $G$ be a doubly transitive group of prime degree $p$.
Then a Sylow $p$-subgroup $P$ of $G$
is strictly contained in its normalizer $N_G(P)$.
\end{lemma}

\begin{proof}
Let $G$ be a doubly transitive group of prime degree $p$
acting on a set $V$.
Consider the action of $G$ on the set $\cal{P}$ of all Sylow $p$-subgroups
of $G$ by conjugation. By Sylow theorems this action is transitive 
with $N=N_G(P)$ as the corresponding stabilizer of $P$. 
If $N=P$ then the intersection of any two stabilizers 
of this action is trivial, and so $G$  acts on $\cal{P}$ as Frobenius group.
It follows that $G$ contains a regular normal subgroup $T$ of order 
$|{\cal P}|\equiv {1\pmod p}$. 
Now consider the action of $T$ on the set $V$. Since $T$ is a normal
subgroup of a transitive group $G$ of prime degree it follows that $T$
is transitive on $V$, a contradiction since $|T|$ is not divisible by $p$.
\end{proof}

\subsection{Derangement  graphs}
\noindent

The intersection density of a permutation group can be studied via
{\em derangements}, that is,  fixed-point-free elements of 
$G$. 
Let ${\cal D}$ be the set of all derangements 
of a permutation group $G$. Then following \cite{MRS21} 
we define the {\em derangement graph} of $G$
to be the graph $\Gamma_G=\Cay(G,{\cal D})$ with vertex set $G$
and edge set consisting of all pairs $(g,h)\in G\times G$ such that
$gh^{-1}\in {\cal D}$. Therefore $\Gamma_G$ is the Cayley graph of $G$
with connection set ${\cal D}$, which is a loop-less simple graph
since ${\cal D}$
does not contain the identity element of $G$ and  
${\cal D}$ is inverse-closed. In the terminology of the
derangement graph an intersecting set of $G$ is an 
independent set or a coclique of $\Gamma_G$.
Since, by a classical theorem of Jordan \cite[Th\'eor$\grave{e}$me~I]{J72},   
a transitive permutation group $G$ on a finite set $V$ 
of cardinality at least $2$
contains   derangements, we have $\rho(G)<|V|$.
(Note also, that by a theorem of Fein, Kantor and Schacher \cite[Theorem 1]{FKS81}, every transitive permutation group contains a derangement of prime power order.)  

The following classical upper bound on the size of the largest coclique in vertex-transitive graphs turns out to be very useful when considering the intersection densities of  permutation groups. Namely, 
the derangement graph $\Gamma_G$ of a permutation group $G$ 
is always vertex-transitive.

\begin{proposition}
\label{pro:coclique}
{\rm \cite{GM16}}
Let $\Gamma$ be a vertex-transitive graph. 
Then the largest coclique in $\Gamma$ is of size  $\alpha(\Gamma)$
bounded by
$$
\alpha(\Gamma)\le\frac{|V(\Gamma)|}{\omega(\Gamma)},
$$
where $\omega(\Gamma)$ is the size of a maximum clique in $\Gamma$.
\end{proposition}

\subsection{Intersection density of transitive groups}
\noindent

\begin{proposition}
\label{pro:1}
Let $G$ be a transitive permutation group  and $\F$ an intersecting
set of $G$. Then there exists an intersecting
set $\F'$ suh that $|\F|=|\F'|$ and $1\in \F'$.
\end{proposition}

\begin{proof}
Take an element $f\in \F$ and let $\F'=f^{-1}\F$.
Then $1\in\F'$ and since for $g_1,g_2\in\F$ the element
$f^{-1}g_1(f^{-1}g_2)^{-1}=f^{-1}g_1g_2^{-1}f$
is not a derangement (as it is a conjugate of a non-derangement)
we can conclude that $\F'$ is an intersecting set of $G$.
\end{proof}

The following observation regarding intersection density of doubly transitive permutation groups was made in \cite{MRS21}.

\begin{proposition}
\label{pro:2-transitive}
{\rm \cite[Lemma~2.1(3)]{MRS21}}
If $G$ is doubly transitive permutation group then $\rho(G)=1$.
\end{proposition}

The following result proved in \cite{MRS21} shows that it suffices to consider
minimal transitive subgroups when searching for the maximum value of ${\cal I}_n$.

\begin{proposition}
\label{pro:minimal}
{\rm \cite[Lemma~6.5]{MRS21}}
If $H\le G$ are transitive permutation groups then
$\rho(G)\le \rho(H)$.
\end{proposition}

\begin{proposition}
\label{pro:semiregular}
Let $G$ be a transitive permutation group acting on a set $V$
and containing a semiregular subgroup $H$
with $k$ orbits on $V$. Then $\rho(G)\le k$.
In particular, if $H$ is regular then $\rho(G)=1$.
\end{proposition}

\begin{proof}
Since $H$ is semiregular it follows that  
for any two different elements $g,h\in H$
the element $gh^{-1}\in H$ is semiregular. 
This implies that $H\subseteq V(\Gamma_G)$
induces a clique in $\Gamma_G$ of size $|H|$. Consequently,
Proposition~\ref{pro:coclique} implies that
$$
\alpha(\Gamma_G)\le\frac{|V(\Gamma_G)|}{|H|}=\frac{|G|}{|H|},
\ \
\textrm{ and so }
\ \ 
\rho(G)=\frac{\alpha(\Gamma_G)}{|G_v|}\le \frac{|G|}{|H||G_v|}=\frac{|V|}{|H|}=k.
$$
If $H$ is regular (that is, if $k=1$) then the above inequality 
gives $\rho(G)\le 1$. But as observed in the introduction the intersection
density is at least $1$ for any permutation group, and so 
we conclude that in this case $\rho(G)=1$.
\end{proof}

%





By the above proposition every transitive permutation group
admitting a regular subgroup has EKR-property.
Trivial examples of permutation groups with the
strict-EKR-property are  regular permutation groups.
Observe  that a
transitive permutation group $G$ admitting a regular subgroup 
of index $2$ also has the strict-EKR-property. 
Namely, if $\F$ is a maximal intersecting set of $G$  containing $1$
and $f\in \F\setminus\{1\}$, then $f$ fixes a point $v$, and therefore $\{1,f\}=G_v$ 
(since stabilizers have order $2$). 
This shows that every generalized dihedral group has the strict-EKR-property. 
The same idea cannot be generalized to cases where $G$ has a regular subgroup of index greater than 2. For example, consider $G=A_4$ acting on $\{1,2,3,4\}$. Then $G$ has the EKR-property, as it admits a regular subgroup $\{id,(1\,2)(3\,4),(1\,3)(2\,4),(1\,4)(2\,3)\}\cong\ZZ_2\times\ZZ_2$ of index 3. However, $G$ does not have the strict-EKR-property since $\{id,(1\,3\,2),(1\,4\,2)\}$ is a maximum non-canonical intersecting set.

In the example below we show that the action of $S_4$ on $2$-element subsets of $\{1,2,3,4\}$ has the EKR-property but not the strict-EKR-property, while the action of $A_4$ on the same set does not have the EKR-property. 

\begin{example}\label{ex:EKR}
{\rm
Let $G=S_4$ acting on the set of all $2$-element subsets of $\{1,2,3,4\}$. Observe that $B=\{\{1,2\},\{3,4\}\}$ is a block of size $2$ for $G$ that induces a complete imprimitivity block system $\B$ with $3$ blocks of size $2$. The kernel of the action of $G$ on $\B$ is $K=\{id,(1\,2)(3\,4),(1\,3)(2\,4),(1\,4)(2\,3)\}$. Observe that $\{id,(1\, 2\, 3\, 4), (1\, 3\, 2), (1\, 4\, 2), (1\, 2\, 4\, 3)\}$ is a clique of size 5 in the derangement graph $\Gamma_G$.
It follows that $\alpha(\Gamma_G)\leq |V(\Gamma_G)|/\omega(\Gamma_G)\leq 24/5$, and since $\alpha(\Gamma_G)$ is an integer, we have 
$\alpha(\Gamma_G)\leq 4=|G_v|$. This shows that $G$ has the EKR-property, that is, $\rho(G)=1$. Observe that $K$ is an intersecting set of size $4$ which is not canonical, and so $G$ does not have the strict-EKR-property.

Also, since $K\le A_4$ it follows that 
the action of $H=A_4$ on the set of all $2$-element subsets of $\{1,2,3,4\}$
has an intersecting set of size $4$, and so $H$ 
does not have the EKR-property. In fact $\rho(H)=2$.

%
}
\end{example}


\section{Transitive groups of prime power degree $p^k$}
\label{sec:2p}
\noindent

The next lemma about intersection densities of transitive permutation groups
admitting imprimitivity block systems arising from semiregular subgroups will be used in the proofs of the main results of this paper.

\begin{lemma}\label{lem:semiq}
Let $G$ be a transitive permutation group 
admitting a semiregular subgroup $H$ whose orbits form 
a $G$-invariant partition
$\B$, and  let $\G$ be the permutation group 
induced by the action of $G$ on $\B$. 
Then $\rho(G)\leq \rho(\G)$.
\end{lemma}

\begin{proof}
Let $G$ be a transitive permutation group acting on a set $V$. 
Let $K=\textrm{Ker}(G\to \G)$ be the kernel of the action of $G$ on $\B$,
and let $\F$ be an intersecting set of $G$. 
We claim that
\begin{eqnarray}
\label{eq:p} 
|\F \cap gH| \leq 1 \textrm{ for every $g\in G$}.
\end{eqnarray}
Let $x,y\in \F \cap gH$. Then $x=gh_1$ and $y=gh_2$ 
for some $h_1,h_2 \in H$, and $xy^{-1}=g(h_1h_2^{-1})g^{-1}$. Since $x,y\in \F$, it follows that $xy^{-1}$ fixes a point. On the other hand, $xy^{-1}$ is a conjugate of an element $h_1h_2^{-1}\in H$, 
and thus since $H$ is semiregular, it follows that $h_1=h_2$, implying that $x=y$, proving (\ref{eq:p}).

We now show that $\overline{\F}=\{\overline{f}\mid f \in F\}$ is an intersecting set of $\G$.  Let $f,g\in \F$. Then $fg^{-1}$ fixes a point $v$, and hence $\overline{fg^{-1}}=\overline{f} \overline{g}^{-1}$ fixes the block of $\B$ that contains $v$, and so $\overline{\F}$ is indeed 
an intersecting set of $\G$.

Let $\overline{f}\in \overline{\F}$ and let $[\overline{f}]=\{g\in \F \mid \overline{g}=\overline{f}\}$ be the set of all those elements in $\F$ whose
image under the homomorphism $G\to\G$ is equal to $\overline{f}$.  
Of course, $[\overline{f}]\subseteq fK$. Writing $fK$ as a union of $|K\colon H|$ cosets of $H$, and using (\ref{eq:p}), it follows that $[\overline{f}]$ contains at most one element from each of the cosets of $H$, that is, $|[\overline{f}]|\leq |K|/|H|$. Since 
$$
\F=\bigcup_{\overline{f}\in \overline{\F}} [\overline{f}]
\textrm{ it follows that }
|\F| \leq \frac{|K||\overline{\F}|}{|H|}.
$$ 

Now $\overline{\F}$ being  an intersecting set of $\G$, implies that
$|\overline{\F}| \leq \rho(\G) \cdot |\G_B|$ for  $B\in \B$. 
Since $\G$ is a transitive permutation group of degree $\frac{|V|}{|H|}$ 
we have that $|\G_B|=\frac{|\G||H|}{|V|}$, and so 
$$
|\F| \leq \frac{|K|}{|H|}\cdot|\overline{\F}| \leq  \frac{|K|}{|H|}\cdot \rho(\G)
\cdot|\G_B|=\frac{|K|}{|H|}\cdot \rho(\G) \cdot \frac{|\G||H|}{|V|}=\rho(\G)\cdot  \frac{|K||\G|}{|V|}=\rho(\G)\cdot |G_v|.
$$
Hence $|\F|/|G_v|\le \rho(\G)$, and 
since $\F$ is an arbitrary intersecting set of $G$, it follows that $\rho(G)\leq \rho(\G)$.
\end{proof}

\begin{proofT2}
Let $G$ be transitive permutation group of degree $p^k$, where $p$ is 
a prime and $k\ge 1$, acting on a set $V$.
Let $P$ be a Sylow $p$-subgroup of $G$ of order  $|P|=p^m$.
Then, by \cite[Theorem~3.4]{W64}, $P$ is transitive on $V$.
In view of Proposition~\ref{pro:minimal} we only need to show that
$\rho(P)= 1$.

The proof will be by  induction on $|P|=p^m$. If $m=1$, it follows that $P$ is regular, hence $\rho(P)=1$ by Proposition~\ref{pro:semiregular}. Suppose that $m>1$, and that intersection density of every transitive $p$-group of order less than $p^m$ is equal to 1.
By a well-known result on $p$-groups,  
the center $Z=Z(P)$ of  $P$ is non-trivial.
Observe that,  since $G$ acts faithfully on $V$, the group
$Z$ is semiregular on $V$. Moreover, $Z$ is a normal subgroup of $P$, hence the orbits of $Z$ form a $P$-invariant partition. Let $Q$ be the permutation group induced by the action of $P$ on the orbits of $Z$. Then $Q$ is a transitive $p$-group of order less than $|P|$, hence by the induction hypothesis $\rho(Q)=1$. Applying Lemma~\ref{lem:semiq} it follows that $\rho(P)\leq \rho(Q)=1$, hence $\rho(P)=1$.
\end{proofT2}


\section{Transitive groups of degree $2p$}
\label{sec:2p}
\noindent

The intersection density of transitive permutation groups of
degree $2p$, $p$ a prime, has first been addressed in \cite{MRS21},
with the partial answer that this density is at most $2$
given in \cite{R21} (see Proposition~\ref{pro:AMC}).
Its proof
relies on the fact that a transitive permutation group of degree $2p$,
$p$ prime, is either doubly transitive, in which case
Proposition~\ref{pro:2-transitive} implies that  its
intersection density equals 1, or it contains a derangement
of order $p$, in which case the corresponding derangement graph
contains a clique of size $p$, and so
Proposition~\ref{pro:coclique} applies
to get that its intersection density is at most $2$.

\begin{proposition}
\label{pro:AMC}
{\rm \cite[Theorem~1.10]{R21}
Let $G$ be transitive permutation group of degree $2p$,
$p$ a prime, then $\rho(G)\le 2$.}
\end{proposition}

Transitive permutation groups of degree $2p$, $p$ a prime, 
have received a considerable attention over the
last decades, mostly within the context of vertex-transitive graphs 
(see  \cite{I1,I2,I3,I4,DM81,S72,W56}).
Such a group is doubly transitive, simply primitive or it admits a 
complete imprimitivity block system consisting of blocks of size $2$ or $p$. 
By Proposition~\ref{pro:2-transitive}  the intersection density of
doubly transitive permutation groups is equal to $1$.
By the classification of finite simple groups (CFSG)
the only simply primitive groups of degree twice a prime 
are the groups
$A_5$ and $S_5$ acting on the set of pairs of a $5$-element set,
see \cite{LS}. 
(It would be of interest to produce a CFSG-free proof of this fact.)

Whereas simply primitive groups and groups admitting a
complete imprimitivity block system consisting of blocks of size $p$
are dealt with directly in the proof of Theorem~\ref{the:main3}, some
preliminary observations are needed for groups admitting a 
complete imprimitivity block system with 
blocks of size $2$.
So let $G$ be a transitive permutation group acting on the set 
$V=\{x_0,\ldots,x_{p-1},y_0,\ldots,y_{p-1}\}$ with 
complete imprimitivity block system $\B$ with blocks $\{x_i,y_i\}$ of size $2$. 
We denote by $\G=G/{\B}$ the permutation group induced by the 
action of $ G$  on $\B$, 
that is, for each $g\in G$ its induced action on $\B$ is denoted by
$\overline{g}$.
In this induced action the block $\{x_i,y_i\}$ is identified
with $i$ for each $i\in \ZZ_p$. 

\begin{lemma}
\label{lem:blocks of size 2 induced action solvable}
Let $p$ be a prime and 
$G$ be a transitive permutation group of degree $2p$ 
acting on a set $V$ and having a 
complete imprimitivity block system $\B$  with blocks 
of size $2$ such that 
the kernel $K=\textrm{Ker}(G\to \G)\ne 1$  
and  the induced action $\G=G/{\B}$ is not doubly transitive. 
Then $\rho(G)=1$ unless
$\G$ is cyclic and $K\le\textrm{Alt}(V)$, 
in which case $\rho(G)=2$. 
\end{lemma}

\begin{proof}
Let $K=\textrm{Ker}(G\to \G)$. If $K$ contains an odd permutation
then it is easy to see that $G$ contains a cyclic regular subgroup.
It follows that $\rho(G)=1$ by Proposition~\ref{pro:semiregular}. 
We may therefore assume that $K\le \textrm{Alt}(V)$.  

Suppose first that $\G$ is cyclic. Then $|G|=|K||\G|=p|K|$. 
Since $K\le \textrm{Alt}(V)$, $K$ contains no derangement,
and so $K$ is an intersecting set.
Since $|K|=2|G_v|$ it follows  that $\rho(G)\geq 2$,
and so, by Proposition~\ref{pro:AMC}, $\rho(G)=2$.

Suppose now that $\G$ is not cyclic. Since $\G$ is not doubly transitive group
of degree $p$, it follows that $\G= \langle a \rangle \rtimes \langle b \rangle \cong \ZZ_p\rtimes \ZZ_d$ for some divisor $d\geq 2$ of $p-1$, where 
we assume that $b$ fixes $0$ and has all other cycles  
of length $d$ in its cycle decomposition. 
Let $\F$ be an intersecting set of $G$. By Proposition~\ref{pro:1} 
we may assume that $1\in \F$, and so no element of $\F$ is a derangement.
In particular, every element of $\F$ must fix at least one block in $\B$, and so we can express $\F$ as a union of disjoint sets
$$
\F=(\F \cap K) \cup \F_0 \cup \F_1 \cup \ldots \cup \F_{p-1},
$$
where $\F_i=\{f\in \F\mid  fix(\overline{f})=\{i\}\}$. Namely,
for $k\in K$ we have $fix(\overline{k})=\ZZ_p$, whereas 
a non-identity element of 
$\G\cong \ZZ_p\rtimes \ZZ_d$ can have at most one fixed point. 
Consequently an element of $G$ can belong to at most one of the sets $\F_i$.

Suppose that $\F_i\neq \emptyset$, for some $i\in \ZZ_p$. 
Since no element of $\F$ is a derangement  
it follows that for each $f\in\F_i$ we have $fix(f)=\{x_i,y_i\}$.
Let $\sigma\in G$ be such that $\overline{\sigma}=b$. 
In particular,  $\sigma\notin K$, and $\overline{\sigma}$ fixes only the block 
$\{x_0,y_0\}$ and is of order $d$. 
Consequently, either $fix(\sigma)=\emptyset$ or $fix(\sigma)=\{x_0,y_0\}$. 
We may  assume that $fix(\sigma)=\{x_0,y_0\}$ for 
if $fix(\sigma)=\emptyset$ we can multiply $\sigma$ with an element of $k\in K$ interchanging $x_0$ and $y_0$
(such an element exists since $K\neq 1$), and so $\overline{\sigma k}=b$ 
and $fix(\sigma k)=\{x_0,y_0\}$.
Choose  $\pi\in G$ in such a way that $\overline{\pi}=a$, and let 
$f_i=\pi^i \sigma \pi^{-i}$. 
Then $fix(f_i)=\{x_i,y_i\}$ and $fix(\overline{f_i})=\{i\}$. Let $K_i=\{k\in K\mid x_i,y_i \in fix(k)\}$.

\medskip
\noindent
{\sc Claim 1:} $\F_i\subseteq f_iK_i\cup f_i^2K_i \cup \ldots \cup f_i^{d-1}K_i$, for every $i\in \ZZ_p$.
\medskip

\noindent
Let $f\in \F_i$ be arbitrary. Then $fix(\overline{f})=\{i\}$, which
 together with the fact that $\G\cong \ZZ_p\rtimes \ZZ_d$
implies that $\overline{f}\in \langle \overline{f_i} \rangle$. 
Therefore $\overline{f}=\overline{f_i}^t$ for some $t\in \ZZ_d$, and so
$\overline{f_i^{-t}f}=1$ in $\G$. Hence 
$f_i^{-t}f \in K$, and so $f\in f_i^tK$.  
Since $f$ is not contained in $K$ we therefore have 
$f=f_i^tk$ for $t\ne 0$ and  $k\in K$. 
Observe that $f$ as an element of $\F_i$ must have
a fixed point, and so $fix(f)=\{x_i,y_i\}$. Consequently, since
 $fix(f_i)=\{x_i,y_i\}$ it follows that
$k$ must also fix $x_i$ and $y_i$. This implies that
$f=f_i^tk\in f_i^tK_i$, completing the proof of Claim 1.

\medskip
\noindent
{\sc Claim 2:} If $\F\cap f_i^tK_i \neq \emptyset$ (with $i\in\ZZ_p$ and $t\in \ZZ_d\setminus\{0\}$), then $\F\cap f_j^tK_j = \emptyset$ for  $j\in \ZZ_p\setminus\{i\}$. 
\medskip

\noindent
Let $f=f_i^tk_i\in \F\cap f_i^tK_i$ for some $k_i\in K_i$, and let 
$g=f_j^tk_j\in \F\cap f_j^tK_j$ for some $j\in \ZZ_p$ and $k_j\in K_j$. Since $\F$ is an intersecting set, it follows that $fg^{-1}$ has a fixed point. 
Since $K$ is normal in $G$ we have 
that $fg^{-1}=f_i^tk_ik_j^{-1}f_j^{-t}=f_i^tf_j^{-t}k$ for some $k\in K$.
Consequently, $\overline{fg^{-1}}=\overline{f_i^tf_j^{-t}}$. Observe that $f_i^tf_j^{-t}= (\pi^i f_0^t \pi^{-i})(\pi^j f_0^{-t} \pi^{-j})=(\pi^i f_0^t \pi^{-i}) \pi^{j-i}(\pi^i f_0^{-t}\pi^{-i})\pi^{i-j}$, which implies that 
$\overline{f_i^tf_j^{-t}}$ belongs 
to the commutator subgroup $[\G,\G]$ of $\G$. But
$[\G,\G]=\langle \overline{\pi} \rangle=\langle a \rangle$.
By assumption $fg^{-1}$ is not a derangement, and so  
it follows that $\overline{fg^{-1}}=1$, that is, $\overline{f}=\overline{g}$. 
Recall that $fix(\overline{f})=fix(\overline{f_i^t})=\{i\}$ and $fix(\overline{g}))=fix(\overline{f_j^t})=\{j\}$. 
It follows that $i=j$, completing the proof of Claim 2.

\medskip
\noindent
{\sc Claim 3:} If exactly $m$ of the sets $\F_i$ are non-empty, then $|\F\cap K|\leq \frac{|K|}{2^m}$.
\medskip

\noindent
Let $W=\{i\in \ZZ_p\mid \F_i\neq \emptyset\}$ and let $|W|=m$. 
Let $f\in \F\cap K$ be arbitrary and take  $g\in \F_i$, $i\in W$. 
Since $g$ is not a derangement we have that $fix(g)=\{x_i,y_i\}$.
But by assumption $gf^{-1}$ must fix a point, and so $fix(gf^{-1})=\{x_i,y_i\}$. Consequently, $f$ must also fix $x_i$ and $y_i$. 
This shows that $\F\cap K\leq K_i$ for every $i\in W$. 
It follows that $\F \cap K \leq \cap_{i\in W}K_i=K_{(W)}$, 
where $K_{(W)}=\{k\in K\mid k(x_i)=x_i \textrm{ for  every } i\in W\}$. 
It is easy to see that $|K_{(W)}|=\frac{|K|}{2^m}$, and so
$|\F\cap K|\leq \frac{|K|}{2^m}$, proving Claim 3.

\medskip

In the rest of the proof we  distinguishing two cases. 
If $\F_i=\emptyset$ for every $i\in \ZZ_p$ then 
$\F\subseteq K$, and so $|\F|\leq |K|\leq \frac{|K|\cdot d}{2}=|G_v|$,
where the second inequality holds since $d\geq 2$.
We conclude that $\rho(G)=1$.

Suppose now that
$\F_i\neq \emptyset$ for some $i\in \ZZ_p$.
Recall that $\F=(\F \cap K) \cup \F_0 \cup \F_1 \cup \ldots \cup \F_{p-1}$. 
By Claim 1 it follows that 
\begin{eqnarray*}
\F\subseteq (\F\cap K) &\cup &  (\F \cap (f_0K_0\cup f_1K_1\cup \ldots \cup f_{p-1}K_{p-1}))\\
&\cup & (\F \cap  (f_0^2K_0\cup f_1^2K_1\cup \ldots \cup f_{p-1}^2K_{p-1}))\\
&\ldots & \\
&\cup & (\F \cap  (f_0^{d-1}K_0\cup f_1^{d-1}K_1\cup \ldots \cup f_{p-1}^{d-1}K_{p-1})).
\end{eqnarray*}
Claim 2 implies that $|\F \cap  (f_0^tK_0\cup f_1^tK_1\cup \ldots \cup f_{p-1}^tK_{p-1})|\leq |f_i^tK_i|=|K_i|=|K|/2$.
Therefore $|\F|\leq |\F\cap K|+(|K|/2)(d-1)$.  
Since  at least one of the sets $\F_i$ is non-empty, Claim 3 implies 
that $|\F\cap K|\leq \frac{|K|}{2}$. Consequently, 
$|\F|\leq \frac{|K|}{2}\cdot d=|G_v|$, and so $\rho(G)=1$. Completing the proof of Lemma~\ref{lem:blocks of size 2 induced action solvable}.
\end{proof}

\begin{corollary}\label{cor:strict-EKR}
Let $G$ be a transitive group
of degree $2p$ acting on a set $V$ that admits a 
complete imprimitivity block system $\B$  with blocks 
of size $2$ such that 
the kernel $K=\textrm{Ker}(G\to \G)\ne 1$  
and  $\G=G/{\B}\cong \ZZ_p\rtimes \ZZ_d$ is not doubly transitive.
If $K\le Alt(V)$ then $G$ has EKR-property if and only if $d>1$ and $G$ has the strict-EKR-property if and only if $d>2$. 
\end{corollary}

\begin{proof}
The claim regarding EKR-property follows directly from Lemma~\ref{lem:blocks of size 2 induced action solvable}, since $G$ has EKR-property if and only if $\rho(G)=1$. 
If $d=2$ then $F=K$ is a maximum intersecting set which is not canonical, implying that $G$ does not have the strict-EKR-property.

Suppose that $d>2$. Then $|K|<|G_v|$, hence a maximum intersecting set cannot be contained in $K$. From the proof of Lemma~\ref{lem:blocks of size 2 induced action solvable}, it follows that the size of an intersecting set $\F$ is at most $\frac{|K|}{2^m}+\frac{|K|}{2}(d-1)$, where $m$ is the number of sets $\F_i$ (defined in the proof of Lemma~\ref{lem:blocks of size 2 induced action solvable}) that are non-empty. 
Observe that $\frac{|K|}{2^m}+\frac{|K|}{2}(d-1)=\frac{d|K|}{2}=|G_v|$ if and only if $m=1$.
It follows that a maximum intersecting set $\F$ equals $K_i\cup f_iK_i \cup \ldots \cup f_i^{d-1}K_i=G_{x_i}$, implying that $G$ has the strict-EKR-property.
\end{proof}

\begin{remark}
\label{remark:strong-EKR}
{\rm  
Note that Example~\ref{ex:EKR} is the special case of the situation described in Corollary~\ref{cor:strict-EKR} with
$p=3$, $d=2$, $|K|=4$ for $S_4$ and $p=3$, $d=1$, $|K|=4$ for $A_4$.
}
\end{remark}

The following result, which can be extracted from 
\cite[Theorem~6.2]{DM81}
and \cite[Lemma~3.4]{MP01}, will be needed in the proof 
of Theorem~\ref{the:main3}.

\begin{proposition} \label{pro:DM2}
{\rm \cite[Lemma~3.4]{MP01}}
Let $G$ be a transitive permutation group of degree 
$2p$, $p$ a prime, admitting  a complete imprimitivity block system $\B$ 
with blocks of size $2$. Then either $G$ also admits blocks of size $p$,
or for any pair $B,B'\in\B$ there exists
$g\in K=\textrm{Ker}(G\to \G)$ fixing $B$ pointwise and $B'$ setwise but not pointwise.
\end{proposition}

We are now ready to prove the main result of this paper.

\medskip

\begin{proofT3}
Let $p$ be an odd prime, $V$ a set of cardinality $2p$ and
$G$ a transitive permutation
group acting on $V$. 
Since every transitive permutation group
of degree $mp$, where $m\le p$,
contains an $(m,p)$-semiregular element (see, for example,
 \cite[Theorem~3.6]{DM81}) it follows that $G$ contains a
$(2,p)$-semiregular automorphism $\pi$.
Let $P=\la\pi\ra$ and let $O$ and $O'$ be the two orbits of $P$.
Applying Proposition~\ref{pro:semiregular} for 
the semiregular subgroup $P$ we   
conclude that $\rho(G)\le 2$ (see also Proposition~\ref{pro:AMC}).

Suppose first that $G$ is primitive. 
Then by CFSG either $G$ is doubly transitive
or $p=5$ and $G$ is isomorphic to $A_5$ or $S_5$
acting on a $10$-element set of pairs of $\{1,2,3,4,5\}$. 
In the first
case $\rho(G)=1$ by Proposition~\ref{pro:2-transitive}. 
As for the second
case it was calculated in \cite{R21} that $\rho(G)=2$ if $G=A_5$
and   $\rho(G)=1$ if $G=S_5$. In fact it can be seen
that for $G=A_5$  
every subgroup $A_4\le A_5$ gives rise to an intersecting set 
of cardinality $12$, forcing $\rho(G)=2$.
On the other hand, if $G=S_5$ then in the associated derangement graph
a clique of size $10$ is obtained from the union of a Sylow $5$-subgroup
and the coset of this subgroup containing an element of order $4$
normalizing this subgroup (see also the more general argument
in the next paragraph). 

Suppose now that $G$ is imprimitive with $\B$ as the corresponding
complete imprimitivity block system.
Clearly, $\B$ either consists of two blocks of size $p$ or 
$p$ blocks of size $2$. 
In the first case $\B=\{O,O'\}$, 
and Lemma~\ref{lem:semiq} implies that
$\rho(G)\le \rho(\G)=1$ where $\G\cong S_2$ is the 
induced action of $G$ on $\B$.

We may therefore assume that $\B$ consists of blocks of size $2$
and that, furthermore, $G$ admits no blocks of size $p$.
Then, by Proposition~\ref{pro:DM2}, 
the kernel $ K=\textrm{Ker}(G\to \G)$ is non-trivial. 
If $\G$ is not doubly transitive, the result follows 
by Lemma~\ref{lem:blocks of size 2 induced action solvable}. 
Namely,  in this case the condition that 
$\G$ is cyclic and that $K\le \textrm{Alt}(V)$ is equivalent
to part (i) of Theorem~\ref{the:main3}.
If $\G$ is doubly transitive then, 
by Lemma~\ref{lem:prime},  
a Sylow $p$-subgroup $\bar{P}$
of $\G$ is strictly contained in $N=N_{\G}(\bar{P})$.
Consequently, the  preimage $H$ of $N$ under the homomorphism
$G\to \G$ is a transitive permutation group of degree $2p$
satisfying the assumptions of Lemma~\ref{lem:blocks of size 2 induced action solvable}. Since $\overline{H}=N$ is not cyclic it follows that
$\rho(H)=1$.
Now Proposition~\ref{pro:minimal} implies that $\rho(G)=1$, too,
completing the proof of Theorem~\ref{the:main3}.
\end{proofT3}


 \end{document}